\documentclass[11pt,reqno]{amsart}
\usepackage{amsmath,amsfonts,amssymb,amsthm,amscd,mathrsfs,mathtools}
\usepackage[T1]{fontenc}
\usepackage[utf8]{inputenc}
\usepackage{geometry}
\usepackage{hyperref}
\hypersetup{colorlinks=true, linkcolor=blue, citecolor=blue, urlcolor=blue}
\usepackage{color}
\usepackage{enumitem}
\usepackage{tikz}
\usepackage{tikz-cd}
\usepackage{mdframed}
\usepackage{parskip}
\setlist[itemize]{leftmargin=*}
\allowdisplaybreaks
\usepackage[normalem]{ulem}

\newcommand{\Z}{\mathbb{Z}}

\newcommand{\C}{\mathbb{C}}

\newcommand{\F}{\mathbb{F}}
\newcommand{\Mat}{\mathrm{Mat}}
\newcommand{\GL}{\mathrm{GL}}
\newcommand{\tr}{\mathrm{tr}}

\usepackage{mathtools}

\DeclarePairedDelimiter\abs{\lvert}{\rvert}%
\DeclarePairedDelimiter\norm{\lVert}{\rVert}%

\makeatletter
\let\oldabs\abs
\def\abs{\@ifstar{\oldabs}{\oldabs*}}
\makeatother

\makeatletter
\let\oldnorm\norm
\def\norm{\@ifstar{\oldnorm}{\oldnorm*}}
\makeatother

\newcounter{definition}[section]
\newenvironment{definition}[2][]{\refstepcounter{definition} \textbf{Definition~\thedefinition #1 -- #2.} \rmfamily}{}

\newcommand{\emphred}[1]{\textcolor{red}{#1}}

\renewcommand\qedsymbol{{\footnotesize\underline{Q.\,E.\,D.}}}
\renewenvironment{proof}{ \medskip  \textbf{Proof.}}{\hspace*{\fill} \qedsymbol \medskip}

\newtheorem{theorem}{Theorem}[section]
\newtheorem{theo}[theorem]{Theorem}
\newtheorem{proposition}[theorem]{Proposition}

\newmdtheoremenv{theoBox}{Theorem}

\title{Spectrum of the Unit-Graph on $\Mat_3(\F_q)$}
\author{Ye\c{s}im Demiro\u{g}lu Karabulut \qquad\\ Heriberto Espinosa}

\begin{document}
\baselineskip=17pt

\begin{abstract}
In this paper, we investigate the spectrum of the unit-graph of the ring of $3\times 3$ matrices over a finite field $\F_q$, which is equivalently the Cayley digraph $\mathrm{Cay}\!\left((\mathrm{Mat}_3(\F_q),+),\mathrm{GL}_3(\F_q)\right)$. This unit-graph has a vertex set $\mathrm{Mat}_3(\F_q)$ with a directed edge from $A$ to $B$ whenever $B-A\in \mathrm{GL}_3(\F_q)$. Then two vertices are adjacent precisely when their difference is invertible. With relevant character theory, we consequently demonstrate that the adjacency spectrum of $\mathrm{Cay}\!\left((\mathrm{Mat}_3(\F_q),+),\mathrm{GL}_3(\F_q)\right)$ consists of four distinct eigenvalues together with their multiplicities. 

Using the Spectral Gap Theorem for Cayley digraphs, we show that if two subsets of vertices in $\mathrm{Mat}_3(\F_q)$ are sufficiently large, then there are matrices in the two subsets whose difference lies in $\mathrm{GL}_3(\F_q)$. In particular, any sufficiently large subset of $\mathrm{Mat}_3(\F_q)$ contains two distinct matrices whose difference has nonzero determinant. This spectral gap implies that large vertex sets cannot avoid each other and must be connected by at least one edge.
\end{abstract}

\maketitle

\section{Introduction and Statement of Results}

Previously, the spectrum---that is, the eigenvalues and their multiplicities---of the adjacency matrix of the unit graph of $\mathrm{Mat}_2(\F_q)$, equivalently $\mathrm{Cay}\big((\mathrm{Mat}_2(\F_q),+),\mathrm{GL}_2(\F_q)\big)$, was determined alongside its spectral gap implications in \cite{KarabulutUnitGraphs}. This naturally motivates the study of cases of higher dimensions $\mathrm{Mat}_n(\F_q)$ for $n>2$, building upon the work established in \cite{KarabulutUnitGraphs} and employing similar approaches and methods. 

Using the discussion and results of Lidl and Niederreiter in \cite{LidlNiederreiterFiniteFields}, their treatment of the absolute trace function, canonical additive characters on finite fields and related character sums avails the extension of additive characters on $\operatorname{Mat}_n(\F_q)$, relying on the trace (map) of matrices in this finite additive group. 

Babai in \cite{Babai_Spectra_of_Cayley_graphs} provides the critical framework for unifying character theory with Cayley graphs, expressing Cayley-graph adjacency spectra in terms of ``irreducible characters'' of a group. This is the theoretic basis for \hyperref[theo:theoremA4]{Theorem A.4} and lays the groundwork on which \hyperref[prop:Prop1.1]{Proposition 1.1} hinge and whose proof and discussion is engaged in Section 4. We set up for \hyperref[prop:Prop1.1]{Proposition 1.1} as follows.

Let $\F_q$ be an arbitrary finite field of order $q$. Let $\mathrm{GL}_3(\F_q)$ be the general linear group of degree $3$ over the finite field $\F_q$, which consists of all $3 \times 3$ invertible matrices with entries in $\F_q$. Consider the unit-graph of the ring of $3\times 3$ matrices over a finite field $\F_q$, which is equivalently the Cayley digraph $\mathrm{Cay}\left((\mathrm{Mat}_3(\F_q),+),\mathrm{GL}_3(\F_q)\right)$. In this paper, we calculate the spectrum of the adjacency matrix of this graph. In particular, we prove the following:

\begin{proposition} \label{prop:Prop1.1}
    Let $B = [b_{ij}] \in \text{GL}_{3}(\F_q)$. The formulae 
    \begin{align*}
        \lambda_{A_0} ={}& \sum_{B \in \text{GL}_{3}(\F_q)}\chi(0) = \abs{\text{GL}_{3}(\F_q)} = q^{9}-q^{8}-q^{7}+q^{5}+q^{4}-q^{3}, \\ 
        \lambda_{A_1} ={}& \sum_{B \in \text{GL}_{3}(\F_q)}\chi(b_{11}) = -q^6 + q^5 + q^4 - q^3, \\
        \lambda_{A_2} ={}& \sum_{B \in \text{GL}_{3}(\F_q)}\chi(b_{11} + b_{22}) = {q^4-q^3},\quad \text{and} \\ 
        \lambda_{A_3} ={}& \sum_{B \in \text{GL}_{3}(\F_q)}\chi(b_{11} + b_{22} + b_{33}) = -q^3
    \end{align*}
    hold for any finite field $\F_q$ and for any non-trivial character $\chi$ on $\F_q$ and are the eigenvalues of the adjacency spectrum of the unit-graph 
    on $\text{Mat}_{3}(\F_q)$.
\end{proposition}

Using the Spectral Gap Theorem for Cayley digraphs (\hyperref[theo:theoremA3]{Theorem A.3}), we show that if two subsets of vertices in $\mathrm{Mat}_3(\F_q)$ are sufficiently large, then there are matrices in the two subsets whose difference lies in $\mathrm{GL}_3(\F_q)$. In particular, we prove the following:

\begin{theorem} \label{theo:theorem1.2}
    If $X,Y\subseteq \Mat_3(\F_q)$ satisfy $\sqrt{\abs{X}\abs{Y}} \;>\; q^6+q^3+2$,
    then there exist $A\in X$ and $B\in Y$ such that $B-A\in \GL_3(\F_q)$. In particular, if $\abs{X}>q^6+q^3+2$, then $X$ contains two distinct matrices whose difference has nonzero determinant. So any set $X$ with $\abs{X}=\Omega(q^6)$ (as $q\to\infty$) cannot avoid such a pair of a matrices.
\end{theorem}    

This spectral gap implies that large vertex sets cannot avoid each other and must be connected by at least one edge.

\section{Background and Setup}

We revisit the equivalence of a Cayley digraph and a unit-graph (\cite{KarabulutUnitGraphs}) and germane definitions to spectral graph theory. Most generally, we can define a \textit{graph} as follows.

\begin{definition}{Graph, Vertices, Edge} A \emphred{graph} $X$ consists of a set of vertices $V(X)$ and a set of edges $E(X)$, where an \emphred{edge} is an unordered pair of distinct \emphred{vertices} of $X$.
\end{definition}

More specifically, we are interested in Cayley digraphs which is the chief concern of this paper.
In particular, for a directed graph---or as a common portmanteau, \textit{digraph} for which there is a directed edge from $x$ to $y$---we can define a Cayley \textit{di}graph.

\begin{definition}{Cayley Digraph}
    Let $H$ be a group with identity and $S$ the group of units in $H$. Then a \emphred{Cayley digraph}, denoted by $\text{Cay}(H,S)$, is defined as the graph with the vertex set $H$, such that if $x,y \in H$, then there exists an edge from $x$ to $y$ in $S$, the edge set of $\text{Cay}(H,S)$, if and only if $y-x \in S$. \cite[p.~11]{Brouwer_Haemers_spectra}
\end{definition}

\begin{definition}{Unit-Graph}
    Let $H$ be a finite ring with identity, and let $S$ denote the set of units. We define the \emphred{unit-graph} $\Gamma$ on $H$ to equal the directed graph whose vertex set is $H$, for which there is a directed edge from $x$ to $y$ if and only if $y-x \in S$. \cite{KarabulutUnitGraphs}
\end{definition}

Note then that a Cayley digraph $\Gamma = \text{Cay}(H,S)$ is equivalent to saying that $\Gamma$ is the unit-graph on $H$.

\begin{definition}{Simple Graph}
    A \emphred{simple graph}, also called a strict graph, is an unweighted, undirected graph containing no graph self-loops, which are edges from a vertex to itself, or multiple edges---two or more edges from a one vertex to another vertex. \cite{WeissteinSimpleGraph} 
\end{definition}

Let $q=p^n$ and consider the additive group $H=(\Mat_3(\F_q),+)$. Let $S=\GL_3(\F_q)$, the set of units in $\Mat_3(\F_q)$. Then the unit-graph on $\Mat_3(\F_q)$ is the Cayley digraph
\[
\Gamma = \mathrm{Cay}\big((\Mat_3(\F_q),+),\GL_3(\F_q)\big),
\]
with an edge between $X$ and $Y$ if and only if $Y-X\in \GL_3(\F_q)$. Our goal is to compute the adjacency spectrum of $\Gamma$ and then apply a spectral gap estimate to large subsets of vertices. In order to do so, we first recall and apply the relevant graph theory and character theory preliminaries. Then we proceed to parameterize all additive characters of $(\Mat_n(\F_q),+)$ to convert the spectral problem into explicit character sums over $\GL_3(\F_q)$, and finally we combine the eigenvalue formulas with the rank-count multiplicities in order to obtain the full spectrum and thereafter analyze its spectral gap consequences.

For the Cayley digraph $\text{Cay}(\text{Mat}_{3}(\F_q), \text{GL}_{3}(\F_q))$, by definition, there exists a (directed) edge from a vertex $X \in \text{Mat}_{3}(\F_q)$ to a vertex $Y \in \text{Mat}_{3}(\F_q)$ if and only if $Y-X \in \text{GL}_{3}(\F_q)$, i.e., if and only if $Y-X$ is invertible. Since this Cayley diagraph can also be understood as an undirected graph (as established in \cite[p.~11]{KarabulutUnitGraphs}), it is the case that $\text{Cay}(\text{Mat}_{3}(\F_q), \text{GL}_{3}(\F_q))$ is an undirected graph. In tandem with \hyperref[prop:PropA2]{Proposition A.2}, this would mean $\text{Cay}(\text{Mat}_{3}(\F_q), \text{GL}_{3}(\F_q))$ is an undirected, simple unit-graph. 

In order to tersely describe whether any pair of vertices in this Cayley digraph possesses an edge between them, we can cleverly inventory the existence of such an edge using conventional binary descriptors of 1 and 0---the former describing if indeed there exists an edge and 0 otherwise---and collect these descriptions of edge existence in a matrix, called the adjacency matrix.

\begin{definition}{Adjacency Matrix}
    Let $G$ be a graph without multiple edges. The \emphred{adjacency matrix} of $G$ is the $0-1$ matrix $\mathbb{A}$ indexed by the vertex set $V(G)$ of $G$ where $A_{xy} = 1$ when there is an edge from $x$ to $y$ in $G$ and $A_{xy} = 0$ otherwise. \cite[p.~1]{Brouwer_Haemers_spectra}
\end{definition}

Consider the unit-graph $\text{Cay}(H,G)$ on $H$. For vertices $h_i,h_j \in H$ for which $i,j \in \abs{H}$, then the adjacency matrix $\mathbb{A} = [a_{ij}]_{\abs{H} \times \abs{H}}$ consists of entries \[a_{ij} = \left\{\begin{array}{cc}
     1    & \text{if } \exists \text{ edge from } h_i \text{ to } h_j\\
     0    & \text{if } \nexists \text{ edge from } h_i \text{ to } h_j\\
    \end{array}\right. = \left\{\begin{array}{cc}
     1    & \text{if } h_j-h_i \in S\\
     0    & \text{if } h_j-h_i \notin S\\
    \end{array}\right.. \]
Then the adjacency matrix of the simple, undirected Cayley digraph $\text{Cay}(\text{Mat}_{3}(\F_q), \text{GL}_{3}(\F_q))$ is $\mathbb{A} = [a_{ij}]_{q^{9} \times q^9}$ where  
\[a_{ij} = \left\{\begin{array}{cc}
     1    & \text{if } h_j-h_i \in \text{GL}_3(\F_q)\\
     0    & \text{if } h_j-h_i \notin \text{GL}_3(\F_q)\\
    \end{array}\right.. \]
Also note that by virtue of the inherent undirected nature of the unit-graph on $\text{Mat}_{3}(\F_q)$ (since the negative of an invertible matrix is still invertible), the adjacency matrix for this Cayley digraph is symmetric. That is, \[\mathbb{A} = [a_{ij}]_{q^{9} \times q^9} =  \left\{\begin{array}{cc}
     1    & \text{if } h_j-h_i \in \text{GL}_3(\F_q)\\
     0    & \text{if } h_j-h_i \notin \text{GL}_3(\F_q)\\
    \end{array}\right. =  \left\{\begin{array}{cc}
     1    & \text{if } h_i-h_j \in \text{GL}_3(\F_q)\\
     0    & \text{if } h_i-h_j \notin \text{GL}_3(\F_q)\\
    \end{array}\right. = [a_{ji}]_{q^{9} \times q^9}. \]

Further, as is common practice with matrices in linear algebra, we can explore the eigenvalues for the adjacency matrix of the Cayley digraph $\text{Cay}(\text{Mat}_{3}(\F_q), \text{GL}_{3}(\F_q))$. This is the exact notion behind the spectrum of a graph. 

\begin{definition}{Spectrum (of a Graph)}
    The \emphred{spectrum} of a graph $X$ is the list of the eigenvalues together with their multiplicities for the adjacency matrix of $X$ (and similarly we refer to the eigenvalues (and their respective multiplicities) of the adjacency matrix of $X$ as the eigenvalues of $X$). \cite[p.~2]{Brouwer_Haemers_spectra}
\end{definition}

Then let us determine the spectrum of the unit-graph on $\text{Mat}_3(\F_q)$. 

\section{\texorpdfstring{Characters on $\F_q$ and on $\Mat_n(\F_q)$}{Characters on Fq and on Matn(Fq)}}

\begin{definition}{Character of an Additive Finite Abelian Group}
Let $G$ be an additive finite abelian group of order $\abs{G}$ with the identity element $0_G$. An \emphred{additive character} $\chi$ of $G$ is a homomorphism from $G$ into the multiplicative group $S^1= \{z \in \C: \abs{z}=1\}$, that is, a mapping $\chi:G \to S^1$ for which $\chi(g_1+g_2) = \chi(g_1)\chi(g_2)$ for all $g_1,g_2 \in G$. \cite[p.~187]{LidlNiederreiterFiniteFields}
\end{definition}

Note since $\chi(0_G) =\chi(0_G + 0_G) = \chi(0_G)\chi(0_G)$, then $\chi(0_G) = 1$.

\begin{definition}{Trivial Character}
    The \emphred{trivial character} $\chi_0$ is defined by $\chi_0(g) = 1$ for all $g \in G$ while any character $\chi_i$ for $i \neq 0$ is called non-trivial. \cite[p.~187]{LidlNiederreiterFiniteFields}
\end{definition}

\bigskip
\begin{definition}{Additive Character of $\F_p$}
    Let $\chi_j: \F_p \to S^1 $ be an \emphred{additive character of} $\F_p$ where $\F_p$ is a prime field. Then it can be shown that $\chi_j$ is of the form $\chi_j(k) = e^{2\pi ijk/p}$ for some fixed $j \in \{0, 1, \cdots p-1 \}$. \cite[p.~187]{LidlNiederreiterFiniteFields}
\end{definition}

The field $\F_q$ for $q = p^n$ is also an additive finite abelian group. The absolute trace function from $\F_q$ to $\F_p$ is defined by
\[\text{Tr}(\alpha) = \alpha + \alpha^p + \alpha^{p^2} + \cdots + \alpha^{p^{n-1}}\]
for every $\alpha \in \F_q$ (\cite[p.~54]{LidlNiederreiterFiniteFields}). Then the canonical additive character of $\F_q$ is then the map $\chi_1 : \F_q \to S^1$ given by
\[\chi_1(\alpha) = e^{2\pi i\, \text{Tr}(\alpha)/p}.\]

\begin{theo}
    For $b \in \F_q$, the map $\chi_b$ with $\chi_b(c) = \chi_1(bc)$ for all $c \in \F_q$ is an additive character of $\F_q$, and every additive character of $\F_q$ is obtained in this way. \cite[p.~190]{LidlNiederreiterFiniteFields}
\end{theo}

Hence, an additive character of $\F_q$ can be re-expressed as $\chi_\alpha(\beta) = \chi_1(\text{Tr}(\alpha\beta)) = e^{2\pi i \textit{Tr}(\alpha\beta)/p}$ for all $\beta \in \F_q$ and some $\alpha \in \F_q$. That is, once we fix the canonical additive character of $\F_q$, then every additive character of $\F_q$ can be generated by scaling by some element of $\F_q$. This yields a family of additive characters for $\F_q$ parameterized by each element of $\F_q$. 

\[
\begin{tikzcd}[column sep=huge, row sep=small]
(\mathbb{F}_q,+) \arrow[r, "{\text{ scaled by } b}"]
& (\mathbb{F}_q,+) \arrow[r, "{\chi_1 \circ \text{Tr}}"]
& S^{1} \\
x \arrow[r, maps to] & b x \arrow[r, maps to]
& e^{2\pi i\, \operatorname{Tr}(b x)/p}
\end{tikzcd}
\]

\bigskip
\begin{proposition}
    If $\chi$ is a non-trivial additive character of $\F_q$, then \[\sum_{\alpha \in \F_q^*}\chi(\alpha) = -1.\]
\end{proposition}

\begin{proof}
    Suppose $\chi$ is a non-trivial additive character of $\F_q$. 
    Then by Theorem 5.4 in \cite[p.~188]{LidlNiederreiterFiniteFields}, $\sum_{\alpha \in \F_q}\chi(\alpha) = 0$. Expanding, $\chi(0) + \sum_{\alpha \in \F_q^*}\chi(\alpha) = 0$. Since $\chi_j(0) = \chi(0) = 1$, then it must be that $\sum_{\alpha \in \F_q^*}\chi(\alpha) = -1$. 
\end{proof}

\underline{Example}. Consider the character $\chi_j: \F_3 \to S^1$ such that $\chi_j(k) = e^{2\pi ijk/3}$ for $k \in \F_3$. From above, we have precisely three characters for $\F_3$ which are $\chi_0$, $\chi_1$, and $\chi_2$. As the trivial character, $\chi_0(k) = 1$ for all $k \in \F_3$. Observe that $\chi_1(k) = e^{2\pi ik/3}$ and $\chi_2(k) = e^{4\pi ik/3}$, so 
\begin{align*}
    \sum_{\alpha \in \F_3^*}\chi_1(\alpha) = \chi_1(1) + \chi_1(2) = e^{\frac{2\pi i}{3}} + e^{\frac{4\pi i}{3}} ={}& \left[\cos\left(\frac{2\pi}{3}\right)+ i\sin\left(\frac{2\pi}{3}\right)\right] + \left[\cos\left(\frac{4\pi}{3} \right)+ i\sin\left(\frac{4\pi}{3}\right)\right]  \\
    ={}& \left(-\frac{1}{2} + \frac{\sqrt{3}}{2}i\right) + \left(-\frac{1}{2} - \frac{\sqrt{3}}{2}i\right) = -1, \\
    \sum_{\alpha \in \F_3^*}\chi_2(\alpha) = \chi_2(1) + \chi_2(2) = e^{\frac{4\pi i}{3}} + e^{\frac{8\pi i}{3}} ={}& \left[\cos\left(\frac{4\pi}{3}\right)+ i\sin\left(\frac{4\pi}{3}\right)\right] + \left[\cos\left(\frac{8\pi}{3} \right)+ i\sin\left(\frac{8\pi}{3}\right)\right]  \\
    ={}& \left(-\frac{1}{2} - \frac{\sqrt{3}}{2}i\right) + \left(-\frac{1}{2} + \frac{\sqrt{3}}{2}i\right) = -1. 
\end{align*}

Similarly, we can define a/an (additive) character of the additive finite abelian group $\text{Mat}_n(\F_q)$.

\begin{definition}{Additive Character of the Finite Group $\text{\normalfont{Mat}}_n(\F_q)$}
    For the trace of a matrix in $\text{Mat}_n(\F_q)$ defined as a map $\text{tr}: \text{Mat}_n(\F_q) \to \F_q$ and an additive character of the finite group $\F_q$ for which $\chi_j \circ \text{Tr}: \F_q \to S^1$, the composition $(\chi_j \circ \text{Tr}) \circ \text{tr}: \text{Mat}_n(\F_q) \to S^1$ is an \emphred{additive character of the finite group} $\text{Mat}_n(\F_q)$ and is defined by $\chi_j(\text{Tr}(\text{tr}(B))) = e^{2\pi i j \text{Tr}(\text{tr}(B))/p}$ for some fixed $j \in \Z_p$ and every $B \in \text{Mat}_n(\F_q)$.
\end{definition}

For $j=1$, we will refer to $\chi_1(\text{Tr}(\text{tr}(B))) = e^{2\pi i \text{Tr}(\text{tr}(B))/p}$ as the \textit{canonical additive character of $\text{\normalfont{Mat}}_n(\F_q)$}. More simply, $\chi_1(\operatorname{tr}(B)) = e^{2\pi i \text{Tr}(\text{tr}(B))/p}$ for $B \in \text{Mat}_n(\F_q)$, or even more directly, we have $\chi_1(B) = e^{2\pi i \text{Tr}(\text{tr}(B))/p}$ for $B \in \text{Mat}_n(\F_q)$.

\begin{proposition}
    For $A \in \text{\normalfont{Mat}}_n(\F_q)$, the map $\chi_A$ with $\chi_A(B) = \chi_1(\text{\normalfont{tr}}(AB))$ for every $B \in \text{\normalfont{Mat}}_n(\F_q)$ is an additive character of $\text{\normalfont{Mat}}_n(\F_q)$, and every additive character of $\text{\normalfont{Mat}}_n(\F_q)$ is generated in this manner.
\end{proposition}

\begin{proof}
    For $B_1, B_2 \in \text{Mat}_n(\F_q)$, we have 
    \begin{align*}
        \chi_A(B_1 + B_2) = \chi_1(\text{tr}(A(B_1 + B_2))) ={}& \chi_1(\text{tr}(AB_1 + AB_2))\\
        ={}& \chi_1(\text{tr}(AB_1) + \text{tr}(AB_2)) \\
        ={}& \chi_1(\text{tr}(AB_1))\chi_1(\text{tr}(AB_2)) = \chi_A(B_1)\chi_A(B_2).
    \end{align*}
    In order to demonstrate the second statement of the theorem, we must show that the sets $A$ and $\chi_A$ for $A \in \text{Mat}_n(\F_q)$ are equipollent, i.e., that the map $A \to \chi_A$ is bijective. Observe that by Theorem 5.5 \cite[p.~189]{LidlNiederreiterFiniteFields}, the number of characters of the finite abelian group $\text{Mat}_n(\F_q)$ is equal to $\abs{\text{Mat}_n(\F_q)}$ which is equal to $q^{n^2}$. In other words, $\abs{A} = \abs{\chi_A}$. It remains to show that map $A \to \chi_A$ is injective and surjective. 

    \indent Suppose that $\chi_A(B) = \chi_{A'}(B)$ for every $B \in \text{Mat}_n(\F_q)$. We must show that $A = A'$. We have $\chi_1(\text{tr}(AB)) = \chi_1(\text{tr}(A'B))$, implying that $\chi_1(\text{tr}(AB))[\chi_1(\text{tr}(A'B))]^{-1} = 1$. Since $\chi_1$ is a homomorphism, then it is the case that $\chi_1(\text{tr}(AB) - \text{tr}(A'B)) = 1$, so $\chi_1(\text{tr}(AB - A'B)) = 1$. Since $\chi_1(\text{tr}(k)) = 1$ if $\text{Tr}(\text{tr}(k)) =0$, then $\text{Tr}(\text{tr}(AB - A'B)) = 0$, or equivalently, we have $\text{Tr}(\text{tr}((A - A')B)) = 0$. Now let $A-A' = [a_{ij}]$ where $a_{ij} \in \F_q$ and fix $B = [b_{ij}]$ where $b_{ij} = 0$ for every $i,j \in \{1,2,\ldots, n\}$ except for some $i$ and $j$ for which $b_{ij}= b$ with some non-zero $b \in \F_q$ (that is, $B$ has only one non-zero entry $b \in \F_q$). Then, $\text{tr}((A - A')B) = a_{ij}b$, so via substitution, $\text{Tr}(a_{ij}b) = 0$. By Theorem 2.23.(ii) \cite[p.~55]{LidlNiederreiterFiniteFields}, $b\text{Tr}(a_{ij}) = 0$. Since $b \neq 0$, it must be that $\text{Tr}(a_{ij}) = 0$, and $a_{ij}=0$ by the non-degeneracy of the absolute trace function. That is, $A-A' = [0]$. Hence, $A = A'$.
    
    \indent Since we have a set of $q^{n^2}$ distinct characters for $\chi_A$ which is equal to all the possible additive characters of $\text{Mat}_n(\F_q)$, an injective map must also be surjective, meaning that every possible character is indeed generated by some matrix $A \in \text{Mat}_n(\F_q)$. That is, the set $\{\chi _{A}\mid A\in \text{Mat}_{n}(\F_{q})\}$ must contain every additive character of $\text{Mat}_{n}(\F_{q})$.
\end{proof}

Hence, for a fixed $A \in \text{Mat}_n(\F_q)$, an additive character of $\text{Mat}_n(\F_q)$ can be re-written as $\chi_{A}(B) = \chi_1(AB) = e^{2\pi i \text{Tr}(\text{tr}(AB))/p}$ for every $B \in \text{Mat}_{n}(\F_{q})$. Then similar to the additive character of $\F_q$, once we fix the canonical additive character of $\F_q$ for the additive character of $\text{Mat}_n(\F_q)$, then every additive character of $\text{Mat}_n(\F_q)$ can be generated by ``scaling'' by some element of $\text{Mat}_n(\F_q)$. This yields a family of additive characters for $\text{Mat}_n(\F_q)$ parameterized by each of the matrix elements of $\text{Mat}_n(\F_q)$.
\[
\begin{tikzcd}[column sep=huge, row sep=small]
(\mathrm{Mat}_n(\mathbb{F}_q), +) 
\arrow[r, "{\text{ multiply by } A}"]
& (\mathrm{Mat}_n(\mathbb{F}_q), +) 
\arrow[r, "{\operatorname{tr}}"]
& (\mathbb{F}_q, +)
\arrow[r, "{\chi_1}"]
& S^{1} \\
B \arrow[r, maps to] 
& AB \arrow[r, maps to]
& \operatorname{tr}(AB) \arrow[r, maps to]
& e^{2\pi i\, \operatorname{Tr}(\operatorname{tr}(AB))/p}
\end{tikzcd}
\]

\section{Eigenvalues from Character Sums}

We introduce the following in order to aid computing the spectrum of the unit-graph on $\text{Mat}_{3}(\F_q)$. By \hyperref[theo:theoremA4]{Theorem A.4}, we can derive the eigenvalues of the adjacency matrix of a unit-graph 
via character sums. That is, for any character on $H$ of the Cayley digraph $\Gamma=\text{Cay}(H,S)$, the associated eigenvalue for its adjacency matrix is $\sum_{s \in S}\chi(s)$.

A key simplification is that the eigenvalue attached to a character $\chi_A(B)=\chi_1(AB)$ depends only on the rank of $A$ which is established in the proof of Proposition 3.2 of \cite{KarabulutUnitGraphs}. Thus matrices of the same rank correspond to the same eigenvalue for the adjacency matrix of $\Gamma=\text{Cay}(H,S)$. We therefore choose the convenient rank representatives
\[
A_0 = \left[\begin{array}{@{\mkern5mu} rrr @{\mkern7mu}}
                  0 & 0 & 0  \\ 
                  0 & 0 & 0  \\
                  0 & 0 & 0  \\
                \end{array}\right],\quad
A_1 = \left[\begin{array}{@{\mkern5mu} rrr @{\mkern7mu}}
                  1 & 0 & 0 \\ 
                  0 & 0 & 0 \\
                  0 & 0 & 0 \\
                \end{array}\right],\quad
A_2 = \left[\begin{array}{@{\mkern5mu} rrr @{\mkern7mu}}
                  1 & 0 & 0 \\ 
                  0 & 1 & 0 \\ 
                  0 & 0 & 0 \\
                \end{array}\right],\quad \text{and} \quad
A_3 = \left[\begin{array}{@{\mkern5mu} rrr @{\mkern7mu}}
                  1 & 0 & 0 \\ 
                  0 & 1 & 0 \\ 
                  0 & 0 & 1 \\
                \end{array}\right].
\]
For $B=[b_{ij}]\in \GL_3(\F_q)$ we then have
\[
\tr(A_0B)=0,\qquad \tr(A_1B)=b_{11},\qquad \tr(A_2B)=b_{11}+b_{22},\qquad \tr(A_3B)=b_{11}+b_{22}+b_{33}.
\]
Consequently, the eigenvalues of the unit-graph on $\Mat_3(\F_q)$ are explicit character sums over $\GL_3(\F_q)$ indexed by the possible ranks $0,1,2$, and $3$.

This enables us to derive four distinct eigenvalues---each associated with the aforementioned ranks---for the adjacency matrix of 
$\mathrm{Cay}\big((\mathrm{Mat}_3(\F_q),+),\mathrm{GL}_3(\F_q)\big)$, all of which are computed explicitly and expressed in \hyperref[prop:Prop1.1]{Proposition 1.1}. Our approach in proving this proposition combines character sums, a detailed case analysis of matrix types, combinatorial arguments, symmetry considerations, and standard linear algebra techniques.

\begin{proof}
    By Proposition 3.2 in \cite{KarabulutUnitGraphs}, we know the unit-graph on $\text{Mat}_{3}(\F_q)$ has at most $4$ distinct eigenvalues. We will show that there are precisely $4$ distinct eigenvalues. Since the eigenvalue depends only on the rank of the matrix parameter $A$ (again, this is established in the proof of Proposition 3.2 of \cite{KarabulutUnitGraphs}), it suffices to work with the rank representatives $A_0,A_1,A_2,A_3$ described above. We can first evaluate $\lambda_{A_0}$ and $\lambda_{A_1}$ directly, then reducing $\lambda_{A_2}$ to a structured count with diagonal constraints, and finally obtaining $\lambda_{A_3}$ from the trace–eigenvalue identity. We follow that same logic here.
    
    \smallskip
    \noindent\textbf{Derivation of $\lambda_{A_0}$.} By \hyperref[theo:theoremA4]{Theorem A.4},
    \begin{align*}
        \lambda_{A_0} = \sum_{B \in \text{GL}_{3}(\F_q)}\chi(\text{tr}(A_0B))
        = \sum_{B \in \text{GL}_{3}(\F_q)}\chi(\text{tr}(\textbf{0}_{3 \times 3}))
        = \sum_{B \in \text{GL}_{3}(\F_q)}\chi(0) 
        ={}& \sum_{B \in \text{GL}_{3}(\F_q)}1 \\
        ={}& \abs{\text{GL}_{3}(\F_q)} \\
        ={}& (q^3-1)(q^3-q)(q^3-q^2) \\
        ={}& q^{9}-q^{8}-q^{7}+q^{5}+q^{4}-q^{3}.
    \end{align*} Note by Equation 1 in \cite[p.~3]{KarabulutUnitGraphs}, $\abs{\text{GL}_{3}(\F_q)} = (q^3-1)(q^3-q)(q^3-q^2)$.

    \smallskip
    \noindent\textbf{Derivation of $\lambda_{A_1}$.} By \hyperref[theo:theoremA4]{Theorem A.4}, 
    \begin{align*}
        \lambda_{A_1} = \sum_{B \in \text{GL}_{3}(\F_q)}\chi(\text{tr}(A_1B)) = \sum_{B \in \text{GL}_{3}(\F_q)}\chi(b_{11}).
    \end{align*}
    Now let $N(\alpha) = \abs{\{B \in \text{GL}_{3}(\F_q): b_{11} = \alpha\}}$ for every $\alpha \in \F_q$. Since each $\chi(b_{11})$ depends solely on $b_{11}$, we can group matrices according to the same value of $b_{11}$. Hence, 
    \[\lambda_{A_1} = \sum_{B \in \text{GL}_{3}(\F_q)}\chi(b_{11}) = \sum_{\alpha \in \F_q}\sum_{B \in \text{GL}_{3}(\F_q) : b_{11} = \alpha}\chi(\alpha) = \sum_{\alpha \in \F_q}\chi(\alpha)\sum_{B \in \text{GL}_{3}(\F_q) : b_{11} = \alpha}1 = \sum_{\alpha \in \F_q}\chi(\alpha) N(\alpha).\]
    By argument of symmetry as all non-zero values of $b_{11}$ occur equally, for every $\alpha \in \F_q^*$, $N(\alpha) = N(1)$. Then 
    \[\lambda_{A_1} = \sum_{\alpha \in \F_q}\chi(\alpha) N(\alpha) = \chi(0)N(0) + N(1)\sum_{\alpha \in \F_q^*}\chi(\alpha) = N(0) - N(1).\] 
    We can compute both $N(0)$ and $N(1)$ by similarly reasoning as with the calculation of the order of the group $\text{GL}_{3}(\F_q)$ which the writer of \cite[p.~3]{KarabulutUnitGraphs} carefully explains in the first paragraph of Section 2.  We can compute $N(0)$ ``by using the fact that a matrix is invertible if and only if its columns (or rows) are linearly independent.'' If $B \in \text{GL}_{3}(\F_q)$ and $B = [v_1 \,\, v_2 \,\, v_3]$ for some column vectors $v_1, v_2, v_3 \in \F_q^3$ where $v_i = (v_{1i}, v_{2i}, v_{3i})$, then $v_1$ can be anything but the zero vector. As $v_{11} = 0$ for $N(0)$, there are $q$ possibilities for $v_{21}$ as well as $q$ possibilities for $v_{31}$ however both cannot be zero or this would result in the zero vector, effectively eliminating one $0$ as a possibility. This means there are $q^2-1$ possibilities for $v_1$. Furthermore, $v_2$ can be anything but a scalar multiple of $v_1$, eliminating $q$ possibilities. Since there are $q$ possibilities each for $v_{12}$, $v_{22}$, and $v_{32}$, this means there are $q^3-q$ possibilities for $v_2$. Lastly, $v_3$ can be anything but a linear combination of $v_1$ and $v_2$, eliminating $q\cdot q = q^2$ possibilities. Since there are $q$ possibilities each for $v_{13}$, $v_{23}$, and $v_{33}$, then this means there are $(q^3-q^2)$ possibilities for $v_3$. By the fundamental counting principle, it follows that $N(0) = (q^2-1)(q^3-q)(q^3-q^2)$. 

    \indent Moving to computing $N(1)$ for which $v_{11} = 1$, then for $v_1$, there are $q$ possibilities each for $v_{21}$ and $v_{31}$ for a total of $q\cdot q = q^2$ options for $v_1$. Via the same reasoning for $v_2$ and $v_3$ as with $N(0)$, there are $q^3-q$ possibilities for $v_2$ and $q^3-q^2$ possibilities for $v_3$ with $N(1)$. Thus, it follows that $N(1) = q^2(q^3 -q)(q^3-q^2)$. Then 
    \[\lambda_{A_1} = N(0) - N(1) = (q^2-1)(q^3-q)(q^3-q^2) - q^2(q^3 -q)(q^3-q^2) = -q^6 + q^5 + q^4 - q^3.\]

    \smallskip
    \noindent\textbf{Derivation of $\lambda_{A_2}$.} By \hyperref[theo:theoremA4]{Theorem A.4},
    \[\lambda_{A_2} = \sum_{B \in \text{GL}_{3}(\F_q)}\chi(\text{tr}(A_2B)) = \sum_{B \in \text{GL}_{3}(\F_q)}\chi(b_{11} + b_{22}).\]
    Now let $N(\alpha,\beta) = \abs{\{B \in \text{GL}_{3}(\F_q): b_{11} = \alpha, b_{22} = \beta\}}$ for every $\alpha, \beta \in \F_q$. Since each $\chi(b_{11} + b_{22})$ depend on the values of $b_{11} + b_{22}$, we can group matrices according to the same values of $b_{11} + b_{22}$. Hence, 
    \begin{align*}
        \lambda_{A_2} = \sum_{B \in \text{GL}_{3}(\F_q)}\chi(b_{11} + b_{22})
        ={}& \sum_{\alpha,\beta \in \F_q} \sum_{B \in \text{GL}_{3}(\F_q): b_{11} = \alpha, b_{22} = \beta}\chi(\alpha + \beta) \\
        ={}& \sum_{\alpha,\beta \in \F_q} \chi(\alpha + \beta)\sum_{B \in \text{GL}_{3}(\F_q): b_{11} = \alpha, b_{22} = \beta}1 \\
        ={}& \sum_{\alpha,\beta \in \F_q} \chi(\alpha + \beta) N(\alpha,\beta).
    \end{align*}
    We also consider that $N(\alpha, \beta)$ depends on whether each $\alpha + \beta$ is zero or non-zero. Thus, 
    \begin{align*}
        \lambda_{A_2} = \sum_{\alpha,\beta \in \F_q} \chi(\alpha + \beta) N(\alpha,\beta) ={}& \sum_{\alpha,\beta \in \F_q: \alpha + \beta = 0} \chi(\alpha + \beta) N(\alpha,\beta) \,\,+ \sum_{\alpha,\beta \in \F_q: \alpha + \beta \neq 0} \chi(\alpha + \beta) N(\alpha,\beta) \\
        ={}& \sum_{\alpha,\beta \in \F_q: \alpha + \beta = 0}N(\alpha,\beta) \,\,+ \sum_{\omega \in \F_q^*}\sum_{ \alpha + \beta = \omega} \chi(\omega) N(\alpha,\beta)
    \end{align*}
    We will consider the possible cases for the first summation $\sum_{\alpha,\beta \in \F_q: \alpha + \beta = 0}N(\alpha,\beta)$. 

    \textbf{Case I}. $\left[\begin{array}{@{\mkern3mu} rrr @{\mkern7mu}}
                  0 & * & *  \\ 
                  * & 0 & *  \\
                  * & * & *  \\
                \end{array}\right]$.

    We then divide into further sub-cases as the first entry of the second row can be zero or non-zero or the second entry of the first row can also be zero or non-zero.
    
    \textbf{Case I}.\textbf{a}.\textbf{i}) 
    $\left[\begin{array}{@{\mkern3mu} rrr @{\mkern7mu}}
                  0 & 0 & *  \\ 
                  0 & 0 & *  \\
                  * & * & *  \\
                \end{array}\right]$.
    However, in this case, the matrices of this form would not be in the $3 \times 3$ general linear group over $\F_q$ since the second column would be a scalar multiple of the first column (or vice-versa), thus eliminating any possible count to the summation.   

    \textbf{Case I}.\textbf{a}.\textbf{ii}) For $\gamma \in \F_q$ such that $\gamma \neq 0$, we have matrices of the form: 
    $\left[\begin{array}{@{\mkern3mu} crr @{\mkern7mu}}
                  0 & 0 & *  \\ 
                  \gamma & 0 & *  \\
                  * & * & *  \\
                \end{array}\right]$.
    If the first entry of the second row is non-zero, considering the second column, there are $q-1$ choices as the third entry of the second column cannot be zero (otherwise the matrices of this form will not be in $\text{GL}_{3}(\F_q)$ since this would clearly yield a determinant of 0). Then for the first column, there are $q(q-1)$ choices which leaves $(q^3-q^2)$ possibilities for the third column. In other words, there are $(q-1)[q(q-1)](q^3-q^2)$ possible matrices of this form for this subcase.   

    \textbf{Case I}.\textbf{b}. For $\gamma \in \F_q$ such that $\gamma \neq 0$, we have matrices of the form: $\left[\begin{array}{@{\mkern3mu} rrr @{\mkern7mu}}
                  0 & \gamma & *  \\ 
                  * & 0 & *  \\
                  * & * & *  \\
                \end{array}\right]$.
    There are $(q^2-1)[q(q-1)](q^3-q^2)$ possible matrices of this form for this subcase.

    \textbf{Case II}. Now for $\alpha \in \F_q^*$, we consider matrices of the following form: $\left[\begin{array}{@{\mkern3mu} rrr @{\mkern7mu}}
                  \alpha & * & *  \\ 
                  * & -\alpha & *  \\
                  * & * & *  \\
                \end{array}\right]$.
    By argument of symmetry, we claim that matrices of this form occur equally as for matrices of the same form with $\alpha = 1$, i.e., for $N(\alpha,-\alpha) = N(1,-1)$. Since this holds for any fixed $\alpha \in \F_q^*$, then to count the number of possible matrices for every non-zero $\alpha \in \F_q$, we must multiply by $q-1$. We then divide into further sub-cases as the first entry of the second row can be zero or non-zero.

    \textbf{Case II}.\textbf{a}. $\left[\begin{array}{@{\mkern3mu} rrr @{\mkern7mu}}
                  1 & * & *  \\ 
                  0 & -1 & *  \\
                  * & * & *  \\
                \end{array}\right]$.
    There are $q(q^2)(q^3-q^2)$ possible matrices of this form for this subcase. Note the second column will never be a scalar multiple of the first column due to the second entries of the first and second column, i.e., $t\cdot 0 = -1$ is not possible for any $t \in \F_q$. 

    \textbf{Case II}.\textbf{b}. For $\gamma \in \F_q$ such that $\gamma \neq 0$, we have matrices of the form: $\left[\begin{array}{@{\mkern3mu} rrr @{\mkern7mu}}
                  1 & * & *  \\ 
                  \gamma & -1 & *  \\
                  * & * & *  \\
                \end{array}\right]$.
    Since $\gamma$ is non-zero, there are $[q(q-1)](q^2-1)(q^3-q^2)$ possible matrices of this form for this subcase. 

    Hence, 
    \begin{align*}
        \sum_{\alpha,\beta \in \F_q: \alpha + \beta = 0}N(\alpha,\beta) ={}& N(0,0) + \sum_{\alpha \in \F_q^*}N(\alpha,-\alpha) \\
        ={}& N(0,0) +(q-1)N(1,-1)
        \\
        ={}& (q-1)[q(q-1)](q^3-q^2)\, + \,(q^2-1)[q(q-1)](q^3-q^2)\, + \\ 
        &\hspace{30mm} (q-1)\left[q(q^2)(q^3-q^2) \, + \, [q(q-1)](q^2-1)(q^3-q^2)\right] \\
        ={}& q^3(q-1)^2(q^3+q^2-1) \\
        ={}& q^8 - q^7 - q^6 + 2 q^4 - q^3.
    \end{align*}
    Now we will consider the possible cases for the remaining summation $\sum_{\omega \in \F_q^*}\sum_{ \alpha + \beta = \omega} \chi(\omega) N(\alpha,\beta)$. 
    
    By a similar argument of symmetry, we can simplify these cases of non-zero $\omega$ with $\omega$ equaling $1$ due to the equal occurrence of each non-zero elements of $\F_q$; we strategically choose the non-zero element $1$ in $\F_q$. Hence, 
    \begin{align*}
        \sum_{\omega \in \F_q^*}\sum_{ \alpha + \beta = \omega} \chi(\omega) N(\alpha,\beta) ={}& \sum_{\omega \in \F_q^*}\chi(\omega)\sum_{ \alpha + \beta = 1}  N(\alpha,\beta) \\
        ={}& \sum_{\omega \in \F_q^*}\chi(\omega)\sum_{ \alpha  \in \F_q} N(\alpha,1-\alpha) \\
        ={}& N(0,1)\sum_{\omega \in \F_q^*}\chi(\omega) + N(1,0)\sum_{\omega \in \F_q^*}\chi(\omega) + \sum_{ \alpha  \in \F_q^*:\alpha \neq 1} N(\alpha,1-\alpha) \sum_{\omega \in \F_q^*}\chi(\omega).
    \end{align*}
    By argument of symmetry, $N(0,1) = N(1,0)$, so we simplify as follows.
    \begin{align*}
        \sum_{\omega \in \F_q^*}\sum_{ \alpha + \beta = \omega} \chi(\omega) N(\alpha,\beta) ={}& N(0,1)\sum_{\omega \in \F_q^*}\chi(\omega) + N(1,0)\sum_{\omega \in \F_q^*}\chi(\omega) + \sum_{ \alpha  \in \F_q^*:\alpha \neq 1} N(\alpha,1-\alpha) \sum_{\omega \in \F_q^*}\chi(\omega) \\
        ={}& 2N(0,1)\sum_{\omega \in \F_q^*}\chi(\omega) + \sum_{ \alpha  \in \F_q^*:\alpha \neq 1} N(\alpha,1-\alpha) \sum_{\omega \in \F_q^*}\chi(\omega) \\
        ={}& -2N(0,1) - \sum_{\alpha \in \F_q^*: \alpha \neq 1} N(\alpha, 1-\alpha).
    \end{align*}
    By analogous reasoning we have already seen, consider the following cases for $N(0,1)$. 

    \bigskip

    \textbf{Case I}. We consider the matrices of the following form: $\left[\begin{array}{@{\mkern3mu} rrr @{\mkern7mu}}
                  0 & * & *  \\ 
                  * & 1 & *  \\
                  * & * & *  \\
                \end{array}\right]$.
    We then divide into further subcases as the first entry of the second row can be zero or non-zero.

    \textbf{Case I}.\textbf{a}.  $\left[\begin{array}{@{\mkern3mu} rrr @{\mkern7mu}}
                  0 & * & *  \\ 
                  0 & 1 & *  \\
                  * & * & *  \\
                \end{array}\right]$.
    There are $(q-1)q^2(q^3-q^2)$ possible matrices of this form for this subcase. Note the second column will never be a scalar multiple of the first column due to the second entries of the first and second column, i.e., $t\cdot 0 \neq 1$ for any $t \in \F_q$. 

    \textbf{Case I}.\textbf{b}. For $\gamma \in \F_q$ such that $\gamma \neq 0$, we have matrices of the form: $\left[\begin{array}{@{\mkern3mu} rrr @{\mkern7mu}}
                  0 & * & *  \\ 
                  \gamma & 1 & *  \\
                  * & * & *  \\
                \end{array}\right]$.
    Since $\gamma$ is non-zero, there are $[q(q-1)](q^2-1)(q^3-q^2)$ possible matrices of this form for this subcase. 
    
    \bigskip
    Now we will consider the possible cases for the remaining summation $\sum_{\alpha \in \F_q^*: \alpha \neq 1} N(\alpha, 1-\alpha)$. 

    \textbf{Case II}. We consider the matrices of the following form: $\left[\begin{array}{@{\mkern3mu} rcr @{\mkern7mu}}
                  \alpha & * & *  \\ 
                  * & 1-\alpha & *  \\
                  * & * & *  \\
                \end{array}\right]$.
    We then divide into further subcases as the first entry of the second row can be zero or non-zero.

    \textbf{Case II}.\textbf{a}. $\left[\begin{array}{@{\mkern3mu} rcr @{\mkern7mu}}
                  \alpha & * & *  \\ 
                  0 & 1-\alpha & *  \\
                  * & * & *  \\
                \end{array}\right]$.
    There are $[q(q-2)](q^2)(q^3-q^2)$ possible matrices of this form for this subcase. As reasoned similarly above, we note that the second column will never be a scalar multiple of the first column. 

    \textbf{Case II}.\textbf{b}. For $\gamma \in \F_q$ such that $\gamma \neq 0$, we have matrices of the form: $\left[\begin{array}{@{\mkern3mu} rcr @{\mkern7mu}}
                  \alpha & * & *  \\ 
                  \gamma & 1-\alpha & *  \\
                  * & * & *  \\
                \end{array}\right]$.
    Since $\gamma$ is non-zero, there are $[(q-2)(q-1)q](q^2-1)(q^3-q^2)$ possible matrices of this form for this subcase. 

    \medskip
    Hence, the sum is re-expressed as follows.
    \begin{align*}
        \sum_{\omega \in \F_q^*}\sum_{ \alpha + \beta = \omega} \chi(\omega) N(\alpha,\beta) ={}& -2N(0,1) - \sum_{\alpha \in \F_q^*: \alpha \neq 1} N(\alpha, 1-\alpha) \\
        ={}& -2[(q-1)q^2(q^3-q^2) + q(q-1)(q^2-1)(q^3-q^2)] \,\,- \\
        &\hspace{20mm}[q(q-2)(q^2)(q^3-q^2) + (q-2)(q-1)q(q^2-1)(q^3-q^2)] \\
        ={}& -q^8 + q^7 +q^6 -q^4.
    \end{align*}

    Thus, the third eigenvalue $\lambda_{A_2}$ is 
    \begin{align*}
        \lambda_{A_2} ={}& \sum_{\alpha,\beta \in \F_q: \alpha + \beta = 0}N(\alpha,\beta) \,\,\,+ \sum_{\omega \in \F_q^*}\sum_{ \alpha + \beta = \omega} \chi(\omega) N(\alpha,\beta) \\
        ={}& (q^8 - q^7 - q^6 + 2 q^4 - q^3) + (-q^8 + q^7 +q^6 -q^4) \\
        ={}& q^3(q-1) = q^4-q^3.
    \end{align*}


    \smallskip
    \noindent\textbf{Derivation of $\lambda_{A_3}$.} At this point, we may be tempted to continue with the same strategy, employing brute force to compute the final eigenvalue $\lambda_{A_3}$. However, since we are computing the eigenvalues of the adjacency matrix of the unit-graph--- which is $\text{Cay}(\text{Mat}_{3}(\F_q), \text{GL}_{3}(\F_q))$---on $\text{Mat}_3(\F_q)$, then we can consider the adjacency matrix itself. Rather than carrying out a comparably long direct count for $\lambda_{A_3}$, we instead exploit the trace--eigenvalue identity from linear algebra: if $\mathbb{A}$ has eigenvalues $\mu_i$ with multiplicities $m_i$, then $\sum_i m_i\mu_i = \tr(\mathbb{A})$ (\hyperref[prop:PropA1]{Proposition A.1}). By \hyperref[prop:PropA2]{Proposition A.2}, since the unit-graph $\text{Cay}(\text{Mat}_{3}(\F_q), \text{GL}_{3}(\F_q))$ on $\text{Mat}_{3}(\F_q)$ is a simple graph, then the adjacency matrix for this graph has a zero diagonal given that simple Cayley digraphs have no loops, indicating $a_{ii} = 0$. In this case since $\text{tr}(\mathbb{A}) = 0$, by \hyperref[prop:PropA1]{Proposition A.1}, \[m_0\lambda_{A_0} + m_1\lambda_{A_1} + m_2\lambda_{A_2} + m_3\lambda_{A_3} = 0,\] where $m_i$ for $i=0,\ldots, 3$ is the multiplicity of $\lambda_{A_i}$, i.e., $m_i$ is the number of matrices in $\text{Mat}_{3}(\F_q)$ with rank $i$ (since matrices of the same rank correspond to the same eigenvalue as already mentioned).
    To that end, observe that clearly, the only matrix of rank $0$ in $\text{Mat}_{3}(\F_q)$ is the zero matrix $[0]_{3\times3}$ and as such $m_0 = 1$. We also note that rank $3$ matrices in $\text{Mat}_{3}(\F_q)$ possess three pivot points, and as such by the Invertible Matrix Theorem, this means we are counting all the invertible matrices in $\text{Mat}_{3}(\F_q)$ which is of course the size of $\text{GL}_3(\F_q)$, i.e., by Equation 1 of \cite{KarabulutUnitGraphs},
    \[m_3 = \abs{\text{GL}_3(\F_q)} = (q^3-1)(q^3-q)(q^3-q^2) = q^{9}-q^{8}-q^{7}+q^{5}+q^{4}-q^{3}.\] 
    In order to calculate $m_1$ and $m_2$, consider Landsberg's formula (\cite[p.~455]{LidlNiederreiterFiniteFields}) for the number of $n\times n$ matrices over $\F_q$ of rank $i$:
    \[m_i = q^{\frac{i(i-1)}{2}}\prod_{k=0}^{i-1}\frac{(q^{n-k}-1)^2}{q^{k+1}-1}.\] Then 
    \[m_1 = \frac{(q^3-1)^2}{q-1} = (q^3-1)(q^2+q+1) = q^{5}+q^{4}+q^{3}-q^{2}-q-1,\] 
    and 
    \[m_2 = q\left[\frac{(q^3-1)^2}{q-1}\cdot\frac{(q^2-1)^2}{q^2-1}\right] = q(q^3-1)(q^2+q+1)(q^2-1) = q^{8}+q^{7}-2q^{5}-2q^{4}+q^{2}+q. \]
    Returning to the equation, $m_0\lambda_{A_0} + m_1\lambda_{A_1} + m_2\lambda_{A_2} + m_3\lambda_{A_3} = 0$, we solve for the fourth eigenvalue $\lambda_{A_3}$: 
    \begin{align*}
        \lambda_{A_3} ={}& -\frac{1}{m_3}\left(m_0\lambda_{A_0} + m_1\lambda_{A_1} + m_2\lambda_{A_2}\right) \\
        ={}& - \frac{1}{(q^3-1)(q^3-q)(q^3-q^2)}\biggl[1\cdot(q^3-1)(q^3-q)(q^3-q^2) \,\,+ \\
        {}&\hspace{5mm} \frac{(q^3-1)^2}{q-1}\cdot[(q^2-1)(q^3-q)(q^3-q^2) - q^2(q^3 -q)(q^3-q^2)] \,\,+ \\
        {}&\hspace{6mm} [q(q^3-1)(q^2+q+1)(q^2-1)]\cdot(q^4-q^3) \biggr] \\
        ={}& -q^3.
    \end{align*}
\end{proof}

\section{Spectrum with Multiplicities}

We now make precise the multiplicity counts alluded to in the proof above. Since the eigenvalues are indexed by rank representatives, these rank counts are exactly the multiplicities of the corresponding eigenvalues.
Let $m_i=\abs{\{A\in \Mat_3(\F_q):\mathrm{rank}(A)=i\}}$. Then we summarize, 
\[
m_0=1 \qquad \quad m_3=\abs{\GL_3(\F_q)}=q^{9}-q^{8}-q^{7}+q^{5}+q^{4}-q^{3}
\]
while
\begin{align*}
m_1 ={}& \frac{(q^3-1)^2}{q-1}
     = q^{5}+q^{4}+q^{3}-q^{2}-q-1, \,\,\text{ and}\\
m_2 ={}& q(q^3-1)(q^2+q+1)(q^2-1)
     = q^{8}+q^{7}-2q^{5}-2q^{4}+q^{2}+q.
\end{align*}

Hence, the final spectrum of the unit-graph on $\Mat_3(\F_q)$ is
\begin{align*}
\bigl(\lambda_{A_0},m_0\bigr)
&=\bigl(q^{9}-q^{8}-q^{7}+q^{5}+q^{4}-q^{3},\,\,1\bigr),\\[1mm]
\bigl(\lambda_{A_1},m_1\bigr)
&=\bigl(-q^6 + q^5 + q^4 - q^3,\,\,
q^{5}+q^{4}+q^{3}-q^{2}-q-1\bigr),\\[1mm]
\bigl(\lambda_{A_2},m_2\bigr)
&=\bigl(q^4-q^3,\,\,
q^{8}+q^{7}-2q^{5}-2q^{4}+q^{2}+q\bigr), \text{ and}\\[1mm]
\bigl(\lambda_{A_3},m_3\bigr)
&=\bigl(-q^3,\,\,
q^{9}-q^{8}-q^{7}+q^{5}+q^{4}-q^{3}\bigr).
\end{align*}

\section{Spectral Gap Application}

As a consequence of the spectral computation, the Spectral Gap Theorem (\hyperref[theo:theoremA3]{Theorem A.3}) for Cayley digraphs reveals important results. In our setting $H=(\Mat_3(\F_q), +)$, $S=\GL_3(\F_q)$, and since $\abs{\lambda_{A_1}} = q^{6} + O(q^5)$ dominates the other magnitudes of the non-trivial eigenvalues $\abs{\lambda_{A_2}} = q^{4} + O(q^3)$ and $\abs{\lambda_{A_3}} = q^3$,  
the maximum non-trivial character sum is $\abs{\lambda_{A_1}}$. Consequently, substituting the spectrum obtained above into the general spectral gap bound yields the Spectral Graph Theorem applied to this setting, expressed in \hyperref[theo:theorem1.2]{Theorem 1.2}. Using polynomial long division, it can be shown that the spectral bound is as follows.

\begin{proof}
    Consider the Cayley digraph $\Gamma = \text{Cay}((\text{Mat}_{3}(\F_q), +),\text{GL}_{3}(\F_q))$. By applying \hyperref[theo:theoremA3]{Theorem A.3} on $\Gamma$, this yields
    \begin{align*}
        n_* = \frac{n}{\abs{S}}\left(\max_{1\leq i \leq n} \norm{\sum_{s \in S} \chi_i(s)}\right) = \frac{\abs{\text{Mat}_{3}(\F_q)}}{\abs{\text{GL}_{3}(\F_q)}}\norm{\lambda_{A_1}} ={}& \frac{q^9\,\abs{-q^6 + q^5 + q^4 - q^3}}
        {(q^3-1)(q^3-q)(q^3-q^2)} \\ 
        ={}& \frac{q^9}{q^3-1} 
        < q^6+q^3+2. 
    \end{align*}
    Consequently, if $q^6+q^3+2 < \sqrt{\abs{X}\abs{Y}}$, then the results follow from the Spectral Gap Theorem for Cayley digraphs (\hyperref[theo:theoremA3]{Theorem A.3}).
\end{proof}

The Spectral Gap Theorem for this Cayley digraph shows if two subsets of vertices in $\mathrm{Mat}_3(\F_q)$ are sufficiently large, then there exists a directed edge from one subset to the other; equivalently, there are matrices in the two subsets whose difference lies in $\mathrm{GL}_3(\F_q)$. In particular, any sufficiently large subset of $\mathrm{Mat}_3(\F_q)$ contains two distinct matrices whose difference has nonzero determinant. This spectral gap implies that large vertex sets cannot avoid each other and must be connected by at least one edge.

\appendix
\section{}

\begin{proposition}[Multiplicity Form of the Trace-Eigenvalue Property] \label{prop:PropA1}
    Let $A \in \C^{n \times n}$. Suppose $A$ has $k$ distinct eigenvalues $\mu_1, \ldots, \mu_k$ with corresponding multiplicities $m_1, \ldots, m_k$. Then the eigenvalues counted with their respective multiplicities sum to \[\sum_{i=1}^{k}m_i\mu_i = \text{tr}(A).\] In particular, if $A$ has a zero diagonal for which $\text{tr}(A) = 0$, then \[\sum_{i=1}^{k}m_i\mu_i = 0.\] 
\end{proposition}

\begin{proof}
    By the Schur decomposition (or by triangularization), there exists a unitary matrix $U$ over $\C$ such that $A = UTU^*$ with upper triangle matrix $T = [t]_{n\times n}$ where $U^*$ is the conjugate transpose of $U$. For eigenvalue $\lambda$, if $T$ is an upper triangular matrix, then $T-\lambda I_n$ is clearly also upper triangular with diagonal entries $t_{jj} - \lambda$. In a triangular matrix, the determinant is the product of its diagonal entries, so \[\det(T-\lambda I_n) = \prod_{j=1}^n(t_{jj}-\lambda).\] Since the eigenvalues of $T$ are exactly the $\lambda$ for which $\det(T-\lambda I_n) = 0$, then by transitivity, $\lambda = t_{jj}$ for some $j = 1, \ldots, n$ as a root of the characteristic polynomial. Hence, the diagonal entries of $T$ are the eigenvalues of $T$ counted with repetition, appearing exactly as many times as its multiplicity. Observe that by the cyclic property of the trace function $\text{tr}$, we have \[\text{tr}(A)= \text{tr}(UTU^*)= \text{tr}(TU^*U)= \text{tr}(T I_n)= \text{tr}(T)= \sum_{j=1}^{n}t_{jj}.\] Grouping equal diagonal entries in $T$ by distinct eigenvalues $\mu_{i}$ with their corresponding multiplicities yields \[\text{tr}(A)= \sum_{j=1}^{n}t_{jj}= \sum_{i=1}^{k}m_i\mu_i.\] Of course, if $A$ has a zero diagonal for which $\text{tr}(A) =0$, then the second result trivially follows.
\end{proof}

\begin{proposition}[Simpleness of the Cayley Digraph] \label{prop:PropA2}
    Let $\Gamma = \text{Cay}(H, S)$ be a Cayley digraph (where by definition $H$ is a group and $S$ is a subset of $H$ that is closed under taking inverses and does not contain the identity of $H$). Then every Cayley digraph $\Gamma$ is simple, possessing no multi-edges or loops.
\end{proposition}

\begin{proof}
    In this proof, we will show that every Cayley digraph $\Gamma$ is simple, i.e., there do not exist multi-edges or loops in a Cayley digraph $\Gamma$. 

    \indent I. Suppose there exists a multi-edge in $\Gamma$, i.e., there are at least two distinct edges from an arbitrary vertex $h_1 \in H$ to another arbitrary vertex $h_2 \in H$. By definition then, $h_2 h_1^{-1} =s_1$  and $h_2h_1^{-1} = s_2$ for some elements $s_1, s_2 \in S$ where $s_1 \neq s_2$. However, by group axioms, $s_1h_1 = h_2 = s_2h_1$, indicating $s_1=s_2$. Then it must be that every edge between any two vertices is unique, i.e., there do not exist multi-edges in a Cayley digraph.

    \indent II. Let $h$ be a vertex of $\Gamma$. Observe that $hh^{-1}$ yields the identity of $H$ which by definition/construction of a Cayley digraph (as noted in the proposition) is not an element of $S$. Hence, by definition, there does not exist an edge from a vertex to itself in a Cayley digraph.  
\end{proof}

\begin{theorem}[Spectral Gap Theorem For Cayley Digraphs] \label{theo:theoremA3}
    Let $\text{Cay}(H,S)$ be a Cayley digraph of order $n$. Let $\{\chi_i\}_{i=0,1,\ldots,n}$ be the set of all distinct characters on $H$ with the trivial character $\chi_0$. Define \[n_* = \frac{n}{\abs{S}}\left(\max_{1\leq i \leq n} \norm{\sum_{s \in S} \chi_i(s)}\right),\] and let $X,Y$ be subsets of vertices of $\text{Cay}(H,S)$. If $\sqrt{\abs{X}\abs{Y}} > n_*$, then there exists a directed edge between a vertex in $X$ and a vertex in $Y$. In particular, if $\abs{X} > n_*$, then there exists at least two distinct vertices of $X$ with a directed edge between them.   
\end{theorem}

\begin{theo}[Eigenvectors and Eigenvalues for Cayley Digraphs from Characters] \label{theo:theoremA4}
Let $\Gamma=\mathrm{Cay}(H,S)$ and let $\mathbb{A}$ be its adjacency matrix.
For any character $\chi$ on $H$, the vector $(\chi(h))_{h\in H}$ is an eigenvector of $\mathbb{A}$
with eigenvalue
\[
\sum_{s\in S}\chi(s).
\]
In particular, the trivial character yields the all-ones eigenvector $\textbf{1}$ with eigenvalue $\abs{S}$.
\end{theo}

\section*{Acknowledgment}
The authors thank Cade Denis O'keefe, an undergraduate student at CSUS, for valuable discussions that contributed to the initial ideas for this project.

\bibliographystyle{amsplain}
\bibliography{references}

\textit{Email address}: \href{mailto:demiroglu@csus.edu}{demiroglu@csus.edu}, \href{mailto:hespinosa@csus.edu}{hespinosa@csus.edu}

\textsc{Department of Mathematics \& Statistics, California State University, Sacramento}

\end{document}